\documentclass[10pt]{article}

\usepackage[top=1in,bottom=1in,left=0.75in,right=0.75in]{geometry}

\usepackage{amsmath,amsfonts,amssymb,amsthm,mathtools}
\usepackage{booktabs,multirow}
\usepackage{hyperref,cleveref}
\usepackage{graphicx,xcolor}
\usepackage{qcircuit,braket}
\usepackage{mymath}

\newtheorem{theorem}{Theorem}
\newtheorem{definition}{Definition}
\newtheorem{lemma}{Lemma}
\newtheorem{example}{Example}

\newcommand\eeye[1][]{\eI_{#1}}

\newcommand\ej[1]{e_{\ttj_#1}}

\begin{document}

\title{\bf Quantum Fourier Transform Revisited}
\author{%
Daan Camps\textsuperscript{1,}%
\footnote{Lawrence Berkeley National Laboratory,
1 Cyclotron Road, Berkeley, CA 94720. Email: \texttt{DCamps@lbl.gov}.} ,
Roel Van Beeumen\textsuperscript{1},
Chao Yang\textsuperscript{1,}}
\date{\small%
\textsuperscript{1}Computational Research Division,
Lawrence Berkeley National Laboratory, CA, United States
}

\maketitle

\abstract{
The fast Fourier transform (FFT) is one of the most successful numerical
algorithms of the 20th century and has found numerous applications in many
branches of computational science and engineering.
The FFT algorithm can be derived from a particular matrix decomposition of the
discrete Fourier transform (DFT) matrix.
In this paper, we show that the quantum Fourier transform (QFT) can be
derived by further decomposing the diagonal factors of the FFT matrix
decomposition into products of matrices with Kronecker product structure.
We analyze the implication of this Kronecker product structure on
the discrete Fourier transform of rank-1 tensors on a classical computer.
We also explain why such a structure can take advantage of an important quantum
computer feature that enables the QFT algorithm to attain an exponential speedup
on a quantum computer over the FFT algorithm on a classical computer. 
Further, the connection between the matrix decomposition of 
the DFT matrix and a quantum circuit is made.  We also discuss 
a natural extension of a radix-2 QFT decomposition to a radix-$d$ QFT decomposition.
No prior knowledge of quantum computing is required to understand
what is presented in this paper. Yet, we believe this paper 
may help readers to gain some rudimentary understanding of the nature of 
quantum computing from a matrix computation point of view.
}

%\keywords{Discrete Fourier Transform, Fast Fourier Transform, Quantum Fourier Transform, Matrix and Tensor Decomposition, Kronecker Product, Quantum Circuits}

\section{Introduction}
The fast Fourier transform (FFT)~\cite{CT1965} is a widely celebrated algorithmic
innovation of the 20th century~\cite{top10}. The algorithm allows us to perform
a discrete Fourier transform (DFT) of a vector of size $N$ in $\bigO(N\log N)$
operations. This is much more efficient than
the $\bigO(N^2)$ operations required in a
brute force calculation in which the $N\times N$ matrix representation of the
DFT is directly multiplied with the vector to be transformed. This efficiency
improvement is quite remarkable and has enabled a wide range of
applications such as signal processing and spectral methods for solving
partial differential equations.
However, the computational complexity of the transform can be further reduced
to $\bigO( (\log N)^2 )$ by using the quantum Fourier transform (QFT) on a
quantum computer.
This is well known in the quantum computing community and is often touted
as an example of the type of exponential speedup a quantum computer
can achieve.  So what makes the QFT so efficient?

In this paper, we will explain what the QFT exactly does, and why it can
achieve an additional exponential factor of speedup. Unlike the derivation 
provided in other references~\cite{Coppersmith1994,NC2010},
our explanation of the QFT requires little knowledge of quantum computing. 
We use the language of matrix and tensor decompositions to 
describe and compare the operations performed in the QFT and FFT.
Just like the FFT, the QFT relies on a special matrix factorization of
the DFT matrix to attain its efficiency.
This factorization produces a product of $\bigO((\log N)^2)$ simpler unitary
matrices that can in turn be written as the sum of 2 Kronecker
products of a $2\times 2$ matrix with $2\times2$ identity matrices.
This derivation of the QFT factorization only requires some knowledge of the
elementary properties of Kronecker products.

We point out that a key distinction between a classical computer and a 
quantum computer makes the QFT decomposition well suited for
a quantum computer.  This is also what makes the QFT much more efficient on
a quantum computer. On the other hand, we will also mention the 
limitation of QFT in terms of its application in classical algorithms
such as fast convolution~\cite{NC2010,Lomont2003}.

This paper is mainly pedagogical in nature.
Apart from giving a matrix and tensor decomposition oriented derivation of the 
QFT algorithm, presenting alternative quantum circuits for the QFT, and
extending the QFT algorithm for qubit based quantum computers
to algorithms for qudit based quantum computers,
it does not contain new results.
However, our presentation should make it easier for researchers in 
traditional numerical computation to understand the features and limitations
of quantum algorithms and the potential of quantum computing.

The paper is structured as follows.
We first review the definition of tensor and Kronecker products
and its properties in section~\ref{sec:tensors}. We then examine
a radix-2 FFT algorithm in section~\ref{sec:fft}, which is applied to
vectors of dimension $N = 2^n$, and describe the matrix factorization
associated with the FFT algorithm.
Similar approaches are used in \cite{Pease1968, Sloate1974,Drubin1971,VL1992}.
In section~\ref{sec:qft}, we show how the matrix decomposition
produced by the FFT algorithm can be modified to produce a
matrix decomposition for the QFT algorithm.
In section~\ref{sec:QC}, we explain why QFT can achieve 
$\bigO((\log N)^2)$ complexity on a quantum computer
and show how the QFT can be more conveniently described by a quantum circuit.
In section~\ref{sec:qudit}, we generalize the radix-2 QFT algorithm 
to a radix-$d$ QFT algorithm and present the quantum circuit for implementing
such a QFT on a $d$-level quantum computer.
The QFT on $d$-level quantum computers was also studied in
\cite{Cao2011, Stroud2002, Dogra2015}.

Throughout the paper, we denote vectors by lowercase Roman characters and
matrices by capital Roman characters, e.g., $v$ and $M$.
Matrices and vectors of exponential dimension $2^n$, and in general $d^n$,
are denoted by, e.g., $\eM_n$ and $\exv_n$, respectively.
For example, $\eeye[n]=\eye[2^n]$ denotes the $2^n \times 2^n$ identity matrix.
The value of a (classical) bit is denoted as $\ttj$ and is either $0$ or $1$ in the
radix-2 case, and $\{0,1,\hdots,d-1\}$ in the radix-$d$ case.

\section{Tensor and Kronecker products}
\label{sec:tensors}

In this section we briefly review the Kronecker product and its properties.
We start by defining the Kronecker product of two matrices.

%%% DEFINITION %%%
\begin{definition}[Kronecker product]
\label{def:kronprod}
The \emph{Kronecker product} of the matrices $A \inC[n][m]$ and $B \inC[p][q]$
is defined as the following $np \times mq$ block matrix
\[
A \otimes B :=
\begin{bmatrix}
a_{11} B & a_{12} B & \cdots & a_{1,m} B \\
a_{21} B & a_{22} B & \cdots & a_{2,m} B \\
\vdots   &  \vdots  & \ddots & \vdots    \\
a_{n,1} B &  a_{n,2} B  & \cdots & a_{n,m} B \\
\end{bmatrix},
\]
where $a_{ij}$ is the $(i,j)$th element of $A$.
\end{definition}

The Kronecker product is a special case of the tensor product, hence,
it is bilinear
\begin{equation}
\begin{aligned}
(\gamma A) \otimes B &= A \otimes (\gamma B) = \gamma (A \otimes B), \\
A \otimes (B + C) &= A \otimes B + A \otimes C, \\
(B + C) \otimes A &= B \otimes A + C \otimes A.
\end{aligned}
\label{eq:kron-prop}
\end{equation}
Another important property is the \emph{mixed-product} identity
\begin{equation}
(A\otimes B)(C\otimes D) = (AC) \otimes (BD)
\label{eq:kron-mixedprod}
\end{equation}
for multiplying Kronecker products of matrices of compatible dimension.
The \emph{direct sum}, not to be confused by the Kronecker sum, is defined as follows.

%%% DEFINITION %%%
\begin{definition}[Direct sum]
\label{def:kronsum}
The \emph{direct sum} of the square matrices $A \inCC{n}$ and $B \inCC{m}$
is defined as the following $(n+m) \times (n+m)$ diagonal block matrix
\[
A \oplus B := \begin{bmatrix} A & \\ & B \end{bmatrix}.
\]
\end{definition}

Rank-$1$ tensors (of vectors) and tensor rank decompositions
will play a key role in the remainder of the
paper.
We therefore formally define them and list some important properties.

%%% DEFINITION %%%
\begin{definition}[Tensor rank decomposition]
\label{def:v-rank1}
Let $v_1,v_2,\ldots,v_n \inC[2]$.
Then the $2^n$ dimensional vector
\[
\bfv_n = v_1 \otimes v_2 \otimes \cdots \otimes v_n,
\]
is a rank-$1$ $n$-way tensor of size $2$ in every mode.
In general, $\exv_n$ is a vector with tensor rank-$r$ if its minimal representation
as a sum of vectors of tensor rank-1 requires $r$ terms,
\begin{equation}
\exv_n = \sum_{i=1}^r v_1\p{i} \otimes v_2\p{i} \otimes \cdots \otimes v_n\p{i}.
\label{eq:cpd}
\end{equation}
\end{definition}
\noindent
The decomposition \eqref{eq:cpd} is known as a tensor rank decomposition or canonical polyadic decomposition
\cite{Hitchcock1927}.

In this paper, we will use the \emph{big-endian} binary convention,
formalized by the following definition.
We also state two lemmas related to binary representations of integers and
tensors.
These results will be used in some proofs later on.

%%% DEFINITION %%%
\begin{definition}[Binary representation]
\label{def:bin}
We define the binary representation of $j \inN: 0 \leq j \leq 2^n-1$ as follows
\[
j = [\ttj_1 \ttj_2 \cdots \ttj_{n-1} \ttj_n]
  = \ttj_1 \cdot 2^{n-1} + \ttj_2 \cdot 2^{n-2} + \cdots +
    \ttj_{n-1} \cdot 2^1 + \ttj_n \cdot 2^0,
%\qquad \ttj_i \in \lbrace 0,1 \rbrace.
\]
where $\ttj_i \in \lbrace 0,1 \rbrace$ for $i = 1,\ldots,n$.
\end{definition}

%%% LEMMA %%%
\begin{lemma}
\label{lem:bin-pow}
For every scalar $\alpha \inC$ and $j \inN: 0 \leq j \leq 2^n-1$, we have
\[
\alpha^j = \alpha^{\ttj_1 2^{n-1}} \alpha^{\ttj_2 2^{n-2}} \cdots
           \alpha^{\ttj_{n-1} 2^{1}} \alpha^{\ttj_n 2^{0}}.
\]
\end{lemma}
%%% PROOF %%%
\begin{proof}
The proof directly follows from \Cref{def:bin}.
\end{proof}

%%% LEMMA %%%
\begin{lemma}
\label{lem:bin-vec}
Let $v_1,v_2,\ldots,v_n \inC[2]$.
Then
\[
\bfv_n := v_1 \otimes v_2 \otimes \cdots \otimes v_n
\qquad \Longleftrightarrow \qquad
\bfv_n(j) = v_1(\ttj_1) v_2(\ttj_2) \cdots v_n(\ttj_n),
\]
where $\bfv_n(j)$ denotes the $j$th element of the vector $\bfv_n \inC[2^n]$.
\end{lemma}
%%% PROOF %%%
\begin{proof}
The proof directly follows from \Cref{def:kronprod,def:bin}.
\end{proof}

We define two elementary matrices that are useful for our analysis. 
%%% DEFINITION %%%
\begin{definition}\label{def:E1E2}
Let $e_1$ and $e_2$ be the first and second column of $\eye[2]$.
Then we define the matrices:
\begin{align*}
E_1 &:= e_1 e_1\T = \begin{bmatrix} 1 & 0\\ 0 & 0 \end{bmatrix}, &
E_2 &:= e_2 e_2\T = \begin{bmatrix} 0 & 0\\ 0 & 1 \end{bmatrix}.
\end{align*}
\end{definition}

We clearly have that $E_1 + E_2 = \eye[2]$.
A particular useful application of \Cref{def:E1E2}
is the decomposition of the direct sum of two matrices
as a sum of two Kronecker product terms:
\begin{equation}\label{eq:sumtokron}
A \oplus B = E_1 \otimes A + E_2 \otimes B.
\end{equation}

\section{Matrix Decomposition for Fast Fourier Transform}
\label{sec:fft}

The discrete Fourier transform maps a series of $N$ complex numbers in
another series of $N$ complex numbers as defined below.

%%% DEFINITION %%%
\begin{definition}[DFT \cite{VL1992}]
\label{ddft}
The discrete Fourier transform of a vector
$x = [x_{0}, \hdots, x_{N-1}]^T \inC[N]$
is defined as the vector
$y = [y_{0}, \hdots, y_{N-1}]^T \inC[N]$
with
\begin{equation}
y_k = \Frac{1}{\sqrt{N}}\Sum_{j=0}^{N-1} \omega_{N}^{kj} \, x_{j},
\qquad k = 0,1,\ldots,N-1,
\label{edft}
\end{equation}
where $\omega_{N} := \exp(\frac{-2\pi\I}{N})$ is an $N$th root of unity.
\end{definition}

Similarly, the inverse discrete Fourier transform (IDFT) is given by
\eqref{edft} where $\omega_N := \exp(\frac{2\pi\I}{N})$ is now
the principal $N$th root of unity.
The expression \eqref{edft} can be written as the multiplication of
a DFT matrix $F_N$ to be defined below with the vector $x$.

% --- The DFT matrix --- %
\subsection{The DFT matrix}
\label{sec:dft}

\Cref{ddft}, or variants thereof with a scalar factor different from $\frac{1}{\sqrt{N}}$\footnote{
The scalar factors of the DFT and the inverse DFT need to multiply to $\frac{1}{N}$. The choice $\frac{1}{\sqrt{N}}$
leads to unitary transformations.},
are most commonly used throughout the literature.
As we focus on the matrix representation and decompositions of that matrix representation,
the following equivalent characterization of the DFT is more useful.
A similar approach is followed in \cite{Pease1968, Sloate1974,Drubin1971,VL1992}.
Van Loan \cite{VL1992} even states: ``I am convinced that life as we know it
would be considerably different if, from the 1965 Cooley--Tukey paper onwards,
the FFT community had made systematic and heavy use of matrix-vector notation!''.

%%% DEFINITION %%%
\begin{definition}[DFT matrix]
\label{def:dft-matrix}
The DFT matrix is defined as the unitary matrix
\begin{equation}
F_N := \Frac{1}{\sqrt{N}}
\begin{bmatrix}
\omega_{N}^{0} & \omega_{N}^{0} & \omega_{N}^{0} & \cdots & \omega_{N}^{0}\\
\omega_{N}^{0} & \omega_{N}^{1} & \omega_{N}^{2} & \cdots & \omega_{N}^{N-1}\\
\omega_{N}^{0} & \omega_{N}^{2} & \omega_{N}^{4} & \cdots & \omega_{N}^{2(N-1)}\\
\vdots & \vdots & \vdots & \ddots & \vdots \\
\omega_{N}^{0} & \omega_{N}^{N-1} & \omega_{N}^{2(N-1)} & \cdots & \omega_{N}^{(N-1)(N-1)}
\end{bmatrix} \inCC{N},
\label{dftmat}
\end{equation}
where $\omega_{N} := \exp(\frac{-2\pi\I}{N})$.
\end{definition}

Using \Defref{def:dft-matrix}, the DFT and IDFT become matrix-vector multiplications $y = F_N x$ and $y = F_N\H x$, respectively.
As an example, the first four DFT matrices are given by:
\[
F_1 = \begin{bmatrix} 1 \end{bmatrix},
\quad
F_2 = \Frac{1}{\sqrt{2}}\begin{bmatrix} 1 & \phantom{-}1\\ 1 & -1\end{bmatrix},
\quad
F_3= \Frac{1}{\sqrt{3}}\begin{bmatrix} 1 & 1 & 1\\ 1 & \omega_{3} & \omega_{3}^{2}\\ 1 & \omega_{3}^{2} & \omega_{3}\end{bmatrix},
\quad
F_4 = \Frac{1}{2}\begin{bmatrix} 1 & \phantom{-}1 & \phantom{-}1 & \phantom{-}1\\ 1 & -\I & -1 & \phantom{-}\I \\ 1 & -1 & \phantom{-}1 & -1\\ 1 & \phantom{-}\I & -1 & -\I \end{bmatrix},
\]
where we used that $\omega_{N}$ simplifies for $N= 1,2,$ and $4$, and the equality $\omega_{N}^m = \omega_{N}^{(m\bmod{N})}$.
The $2 \times 2$ DFT matrix $F_2$ is also called the \emph{Hadamard} matrix and denoted as $H$.

In order to simplify notation, we will make use of the exponent notation
for the DFT matrix as introduced by Pease \cite{Pease1968}:
\[
\Ftil_{N} =
\begin{bmatrix}
0 & 0 & 0 & \cdots & 0\\
0 & 1 & 2 & \cdots & N-1\\
0 & 2 & 4 & \cdots & 2(N-1)\\
\vdots & \vdots & \vdots & \vdots & \vdots \\
0 & N-1 & 2(N-1) & \cdots & (N-1)(N-1)
\end{bmatrix},
\]
where only the exponents of $\omega_N$ are listed in the matrix $\Ftil_N$.
For the remainder of the paper, all exponent matrices will be indicated
by a tilde.
Note that by using exponent notation, an multiplication of $\omega_N$ factors becomes an addition of its exponents.

% --- Radix-2 decomposition --- %
\subsection{Radix-$2$ decomposition of the discrete Fourier transform matrix}\label{r2}

A straightforward implementation of the (inverse) DFT as a matrix-vector product
requires $\bigO(N^2)$ operations.
The radix-2 FFT algorithm, popularized by Cooley and Tukey \cite{CT1965}, computes the DFT for vectors of
length $N = 2^n$ in only $\bigO(N \log_2{N}) = \bigO(2^n n)$ operations.
This speedup is achieved by a divide-and-conquer approach which relates a permuted DFT matrix of dimension $N$ to
a permuted DFT matrix of size $N/2$.
Let us illustrate this matrix decomposition for the DFT matrix of dimension $8$.

%%% EXAMPLE %%%
\begin{example}\label{DFT8}
The exponent DFT matrix of dimension $8$ is given by:
\[
\Ftil_8 =
\begin{bmatrix}
0 & 0 & 0 & 0 & 0 & 0 & 0 & 0\\
0 & 1 & 2 & 3 & 4 & 5 & 6 & 7\\
0 & 2 & 4 & 6 & 8 & 10 & 12 & 14\\
0 & 3 & 6 & 9 & 12 & 15 & 18 & 21\\
0 & 4 & 8 & 12 & 16 & 20 & 24 & 28\\
0 & 5 & 10 & 15 & 20 & 25 & 30 & 35\\
0 & 6 & 12 & 18 & 24 & 30 & 36 & 42\\
0 & 7 & 14 & 21 & 28 & 35 & 42 & 49\\
\end{bmatrix}
\stackrel{\bmod{8}}{\equiv}
\begin{bmatrix}
0 & 0 & 0 & 0 & 0 & 0 & 0 & 0\\
0 & 1 & 2 & 3 & 4 & 5 & 6 & 7\\
0 & 2 & 4 & 6 & 0 & 2 & 4 & 6\\
0 & 3 & 6 & 1 & 4 & 7 & 2 & 5\\
0 & 4 & 0 & 4 & 0 & 4 & 0 & 4\\
0 & 5 & 2 & 7 & 4 & 1 & 6 & 3\\
0 & 6 & 4 & 2 & 0 & 6 & 4 & 2\\
0 & 7 & 6 & 5 & 4 & 3 & 2 & 1\\
\end{bmatrix}\ \,
\begin{matrix}
000\\
001\\
010\\
011\\
100\\
101\\
110\\
111\\
\end{matrix}\ \ ,
\]
where in the matrix on the right the exponents are given modulo $8$.
The binary representation of the row indices are shown on the right side of the matrix.
Permuting the rows such that the binary representation of the row indices reverse in order results
in the matrix $\Ftil_{8}^{\prime}$:
\[
\Ftil_{8}^{\prime} =
\left[\begin{array}{c|c}
\begin{matrix}
  0 & 0 & 0 & 0 \\
  0 & 4 & 0 & 4 \\
  0 & 2 & 4 & 6 \\
  0 & 6 & 4 & 2
\end{matrix} & \begin{matrix}
  0 & 0 & 0 & 0 \\
  0 & 4 & 0 & 4 \\
  0 & 2 & 4 & 6 \\
  0 & 6 & 4 & 2
\end{matrix} \\ \hline
\begin{matrix}
  0 & 1 & 2 & 3 \\
  0 & 5 & 2 & 7 \\
  0 & 3 & 6 & 1 \\
  0 & 7 & 6 & 5
\end{matrix} & \begin{matrix}
  4 & 5 & 6 & 7 \\
  4 & 1 & 6 & 3 \\
  4 & 7 & 2 & 5 \\
  4 & 3 & 2 & 1
\end{matrix}
\end{array}\right]
\begin{array}{c}
  000 \\
  100 \\
  010 \\
  110 \\ \hline
  001 \\
  101 \\
  011 \\
  111
\end{array}
=
\left[\begin{array}{c|c}
\Ftil_{4}^{\prime} & \Ftil_{4}^{\prime}\\
\hline
\Ftil_{4}^{\prime} + \Omegat_{4} & \Ftil_{4}^{\prime} + \Omegat_{4} + 4\\
\end{array}\right],
\]
with
\[
\Ftil_{4}^{\prime} = \begin{bmatrix} 0 & 0 & 0 & 0\\ 0 & 2 & 0 & 2\\ 0 & 1 & 2 & 3\\ 0 & 3 & 2 & 1 \end{bmatrix}, \qquad
\Omegat_4 = \begin{bmatrix} 0 & 1 & 2 & 3\\ 0 & 1 & 2 & 3\\ 0 & 1 & 2 & 3\\ 0 & 1 & 2 & 3 \end{bmatrix}.
\]
The left side of the above expression shows a block partitioning of $\Ftil^{\prime}_8$
in terms of the bit-reversed exponent DFT matrix of half the dimension, $\Ftil_{4}^{\prime}$.
Here we again applied the modulo $8$ equivalence
and the property that the entries in $\Ftil_{4}^{\prime}$ must be doubled when used in $\Ftil_{8}^{\prime}$ because $\omega_4 = \omega_{8}^2$.
This partitioning can be written in terms of the permuted DFT matrices as:
\[ F^{\prime}_8 = \frac{1}{\sqrt{2}}
\left[\begin{array}{c|c}
F^{\prime}_4 & F^{\prime}_4\\
\hline
F^{\prime}_4 \Omega_4 & -F^{\prime}_4\Omega_4
\end{array}\right],
\qquad \text{with} \qquad \Omega_{4} =\diag(1,\omega_{8}^1,\omega_{8}^2,\omega_{8}^3),
\]
where the minus sign in the $(2,2)$-block comes from $\omega_{8}^{4} = -1$.
\end{example}

The generalization of the permutation and partition to a DFT matrix of
dimension $2^n$ for $n>2$ requires the definition of $\Omega_{2^n}$, which
we give below.

%%% DEFINITION %%%
\begin{definition}
\label{def:bOn}
Define $\bOmega_n$ as the following $2^{n} \times 2^{n}$ diagonal matrix: 
\[
\eOmega_{n} = \Omega_{2^{n}} :=
\begin{bmatrix} \omega_{2^{n+1}}^{0} & & & \\ & \omega_{2^{n+1}}^{1} & & \\ & & \ddots & \\ & & & \omega_{2^{n+1}}^{2^{n}-1} \end{bmatrix},
\]
where $\omega_{2^{n+1}} := \exp(\frac{-2\pi\I}{2^{n+1}})$.
\end{definition}

The block partitioning of the $4\times 4$ DFT matrix that we demonstrated 
in \Cref{DFT8} can be generalized for any $2^n \times 2^n$ DFT matrix as stated in the following theorem.

%%% THEOREM %%%
\begin{theorem}[See \cite{Sloate1974,Pease1968}]\label{thm:r2split}
Let $\eF_{n}^{\prime} = \eP_n \eF_n$ be the $2^n \times 2^n$ \emph{bit-reversed} DFT matrix,
where $\eP_n$ is the bit reversal permutation matrix acting on $n$ bits.
Then $\eF_{n}^{\prime}$ admits the following factorization:
\begin{equation}\label{eq:r2split}
\begin{aligned}
\eF_{n}^{\prime}
& = \Frac{1}{\sqrt{2}}
\begin{bmatrix} \eF_{n-1}^{\prime} & \eF_{n-1}^{\prime}\\ \eF_{n-1}^{\prime} \eOmega_{n-1} & - \eF_{n-1}^{\prime} \eOmega_{n-1} \end{bmatrix},
\\[5pt]
& =
\begin{bmatrix} \eF_{n-1}^{\prime} & \\ & \eF_{n-1}^{\prime} \end{bmatrix}
\begin{bmatrix} \eeye[n-1] &  \\ & \eOmega_{n-1} \end{bmatrix}
\left(\Frac{1}{\sqrt{2}}\begin{bmatrix}
\eeye[n-1] & \phantom{-}\eeye[n-1] \\
\eeye[n-1] & -\eeye[n-1]\end{bmatrix}\right),
\\[5pt]
& =
(\eye[2] \otimes \eF_{n-1}^{\prime})
(\eeye[n-1] \oplus \eOmega_{n-1})
(H \otimes \eeye[n-1]),
\end{aligned}
\end{equation}
where $\eOmega_{n-1}$ is given by \Cref{def:bOn}.
\end{theorem}

The subdivision of the permuted DFT matrix in a $2 \times 2$ block matrix is also called a \emph{radix-$2$}
splitting in the literature.
The bit reversal permutation $\eP_n$ satisfies the following identity
\begin{equation}
\eP_n (v_1 \otimes \cdots \otimes v_n) = v_n \otimes \cdots \otimes v_1.
\label{eq:bitrev}
\end{equation}
where $v_k \inC[2]$ for $k = 1, \cdots, n$. The matrix $\eP_n$ 
is involutary, because $\eP_{n}^2 = \eeye[n]$, i.e., reversing the bits 
twice gives the original ordering. It is also unitary because it is a 
permutation matrix.

\Cref{thm:r2split} can be applied repeatedly to obtain the complete radix-2 factorization of the DFT matrix.
This is formalized in the following theorem.

%%% THEOREM %%%
\begin{theorem}\label{thm:r2fft}
The $2^n \times 2^n$ DFT matrix $\eF_n$ can be factored as:
\begin{equation}\label{er2fft}
\eF_{n} = \eP_n \eF_{n}^{\prime} = \eP_n \eA^{(0)}_n \eA^{(1)}_n \cdots \eA^{(n-1)}_{n},
\end{equation}
where, for $k=0,1, \hdots, n-1$,
\begin{equation}\label{er2fft2}
\eA^{(k)}_{n} = \eeye[n-k-1] \, \otimes \eB_{k+1},
\qquad \text{and} \qquad
\eB_{k+1}
= \Frac{1}{\sqrt{2}} \begin{bmatrix} \eeye[k] & \eeye[k] \\ \eOmega_{k} & -\eOmega_{k} \end{bmatrix}
= (\eeye[k] \oplus \eOmega_k)(H \otimes \eeye[k]),
\end{equation}
with $\eOmega_{k}$ given by \Cref{def:bOn}.
\end{theorem}
%%% PROOF %%%
\begin{proof}
We give the proof for $\eF_{n}^{\prime}$ by induction based on \Cref{thm:r2split}.
The case $n=1$ can be directly verified:
\[
\eF_{1}^{\prime} = \eA_{1}^{(0)} = 1 \otimes \eB_1 = \frac{1}{\sqrt{2}}  \begin{bmatrix} 1 & \phantom{-}1\\ 1& -1 \end{bmatrix}.
\]
Assume $\eF^{\prime}_n = \eA^{(0)}_{n} \eA^{(1)}_{n} \cdots \eA^{(n-1)}_{n}$ as prescribed by \Cref{thm:r2fft}.
For the bit-reversed DFT matrix of size $2^{n+1}$ we have by \Cref{thm:r2split},
\begin{align*}
\eF^{\prime}_{n+1}
 &= (\eye[2] \otimes \eF_n^{\prime})
    (\eeye[n] \oplus \eOmega_n) (H \otimes \eeye[n]),\\
 &= \left(I_2 \otimes  \left[\eA_n\p0 \eA_n\p1 \cdots \eA_n\p{n-1}\right] \right)
    (\eeye[n] \oplus \eOmega_n) (H \otimes \eeye[n]),\\
 &= \left(I_2 \otimes \eA_n\p0\right) \left(I_2 \otimes \eA_n\p1\right) \cdots
    \left(I_2 \otimes \eA_n\p{n-1}\right) \eA_{n+1}^{(n)},\\
 &= \eA_{n+1}\p0 \eA_{n+1}\p1 \cdots \eA_{n+1}\p{n-1} \eA_{n+1}\p{n},
\end{align*}
where we used $\eye[2] \otimes \eA_n\p{k} = \eA_{n+1}\p{k}$ for
$k = 0,1,\ldots,n-1$, and $\eA_{n+1}\p{n} = \eB_{n+1}$.
\end{proof}

They key observation here is that each matrix $\eA_{n}^{(i)}$ in \eqref{er2fft2} only has two
nonzero elements on every row.
Consequently, the matrix-vector product $\eA_{n}^{(i)} v$ can be computed in only $\bigO(2^n)$ operations,
resulting in an overall $\bigO(2^n n)$ or $\bigO(N \log N)$ computational complexity for the Fourier transform of a vector of size $N = 2^n$.
This is why the FFT is called \emph{fast}.

% --- QFT decomposition of diagonal matrices --- %
\section{Matrix Decomposition for Quantum Fourier Transform}\label{sec:qft}
As we can see from the previous section, the reduction in complexity
from $\bigO(N^2)$ to $\bigO(N\log N)$ in the FFT algorithm 
essentially results from the Kronecker product structure that appears in
the matrix factors of $\eF^{\prime}_{n}$ in \eqref{eq:r2split}.
It is the $\eeye[k] \oplus \eOmega_k$ factor that retains the $N$ factor in
the complexity of the FFT algorithm because the multiplication of
that diagonal matrix with a vector has to be performed in $\bigO(N)$
operations.

It is conceivable that the complexity of the computation can potentially
be reduced further, at least in some special cases, if we can somehow rewrite 
this diagonal matrix as a
Kronecker product of $2 \times 2$ matrices.  It turns out that we can almost do
that.  In this section, we show that the diagonal matrix
$\eeye[k] \oplus \eOmega_k$ can be written as a product of
$k$ simpler matrices, each of which is the sum of two
Kronecker products of $2\times 2$ matrices.  This factorization
yields a decomposition of the DFT matrix that enables
the DFT to be performed efficiently on a quantum computer. 
Exactly how that is achieved will be discussed in the next section. 
Here we will simply refer to this decomposition as the decomposition 
used by the quantum Fourier transform (QFT).

We begin by showing that the diagonal matrix
$\eOmega_n$ can be written as a Kronecker product of $2 \times 2$ unitary matrices
of the following form.

%%% DEFINITION %%%
\begin{definition}
\label{def:Rn}
Define
\begin{equation}
R_n := \begin{bmatrix} \omega_{2^n}^{0} & \\ & \omega_{2^n}^{1} \end{bmatrix}
 = \begin{bmatrix} 1 & \\ & \omega_{2^n} \end{bmatrix},
\label{eq:Rn}
\end{equation}
where $\omega_{2^n} := \exp(\frac{-2\pi\I}{2^n})$.
\end{definition}

The matrix $R_n$ satisfies the identity stated in the following lemma.

%%% LEMMA %%%
\begin{lemma}
\label{lem:Rn}
Let $R_n$ be defined by \Cref{def:Rn}. Then
\[
R_n^{2^j} = R_{n-j},
\]
for $j \inN$.
\end{lemma}
\begin{proof}
We start from the following identity
\[
\omega_{2^n}^{2^j}
 = \left( \exp\left(\frac{-2\pi\I}{2^n}\right) \right)^{2^j}
 = \exp\left(\frac{-2\pi\I}{2^{n-j}}\right)
 = \omega_{2^{n-j}}.
\]
Hence,
\[
(R_n)^{2^j} =
\begin{bmatrix}
1 & \\ & \omega_{2^n}^{2^j}
\end{bmatrix} =
\begin{bmatrix}
1 & \\ & \omega_{2^{n-j}}
\end{bmatrix} = R_{n-j}.
\]
\end{proof}

We can use this result to decompose $\eOmega_n$ as a
Kronecker product of $n$ $R_i$ matrices.

%%% LEMMA %%%
\begin{lemma}\label{lemma:Omega}
Let $\bOmega_n$ be defined by \Cref{def:bOn}.
Then 
\begin{equation}\label{eq:kronfact}
\eOmega_n = R_2 \otimes R_3 \otimes \cdots \otimes R_n \otimes R_{n+1},
\end{equation}
where $R_n$ is defined by \Cref{def:Rn}.
\end{lemma}
\begin{proof}
From \Cref{lem:bin-pow}, we can use the binary representation of $j$ to
rewrite the $j$th diagonal element of $\eOmega_n$ as
\begin{equation}\label{eq:O-bin-pow}
\eOmega_{n}(j,j) = \omega_{2^{n+1}}^{j} 
= \omega_{2^{n+1}}^{\ttj_1 2^{n-1}} \omega_{2^{n+1}}^{\ttj_2 2^{n-2}} \cdots \omega_{2^{n+1}}^{\ttj_{n-1} 2^{1}} \omega_{2^{n+1}}^{\ttj_{n} 2^{0}},
\end{equation} 
for $j \inN : 0 \leq j \leq 2^{n}-1$.
We can rewrite \eqref{eq:O-bin-pow} as
\begin{equation}\label{eq:O-to-G}
\eOmega_{n}(j,j) = R_{n+1}^{2^{n-1}}(\ttj_1,\ttj_1) R_{n+1}^{2^{n-2}}(\ttj_2,\ttj_2) \cdots R_{n+1}^{2^1}(\ttj_{n-1},j_{n-1}) R_{n+1}^{2^0}(\ttj_{n},\ttj_{n}).
\end{equation}
It follow from \Cref{lem:Rn} that \eqref{eq:O-to-G} simplifies to
\begin{equation}
\label{eq:O-to-R}
\eOmega_{n}(j,j) = R_2(\ttj_1,\ttj_1) R_3(\ttj_2,\ttj_2) \cdots R_n(\ttj_{n-1},j_{n-1}) R_{n+1}(\ttj_{n},\ttj_{n}),
\end{equation}
such that \eqref{eq:kronfact} follows from \Cref{lem:bin-vec},
which extends trivially to diagonal matrices.
\end{proof}

Our objective is to decompose $\eeye[k] \oplus \eOmega_k$ into a product of
$k$ matrices, each of which can be written as the sum of Kronecker 
products of $2\times 2$ matrices.  The follow theorem and its proof
shows how this can be done.

%%% THEOREM %%%
\begin{theorem}\label{thm:r2-diagdecomp}
The diagonal operators $\eeye[k] \,\oplus\, \eOmega_k \inCC{2^{k+1}}$ from \Cref{thm:r2fft}
admit the decomposition
\begin{equation}\label{eq:dd}
\eeye[k] \oplus \eOmega_k
= \prod_{i=1}^{k} \Big[
      E_1 \otimes \eeye[i-1] \otimes \eye[2] \otimes \eeye[k-i] \ + \
      E_2 \otimes \eeye[i-1] \otimes R_{i+1} \otimes \eeye[k-i]\Big],
\end{equation}
where $R_i$ is defined in \Cref{def:Rn}.
\end{theorem}
%%% PROOF %%%
\begin{proof}
We start by splitting the diagonal matrix $\eeye[k] \oplus \eOmega_k$
into 2 terms using \eqref{eq:sumtokron}, which yields
\begin{equation}
\eeye[k] \oplus \eOmega_k = E_1 \otimes \eeye[k] + E_2 \otimes \eOmega_k.
\label{eq:lhsresult}
\end{equation}
Next, we rewrite these terms in a redundant form based on the mixed product
identity of Kronecker products \eqref{eq:kron-mixedprod}.
Using the property $E_1 = E_1^2$, we get for the first term
\begin{align*}
E_1 \otimes \eeye[k]
 &= E_1 \otimes  \eye[2] \otimes \cdots \otimes \eye[2],\\
 &= (E_1 \otimes \eye[2] \otimes \eye[2] \otimes \cdots \otimes \eye[2])
    (E_1 \otimes \eye[2] \otimes \eye[2] \otimes \cdots \otimes \eye[2])
    \cdots
    (E_1 \otimes \eye[2] \otimes \eye[2] \otimes \cdots \otimes \eye[2]),\\
 &= \prod_{i=1}^{k} E_1 \otimes \eeye[i-1] \otimes \eye[2] \otimes \eeye[k-i],
\end{align*}
and using \Cref{lemma:Omega} and the property $E_2 = E_2^2$, the second term
results in
\begin{align*}         
E_2 \otimes \eOmega_k
 &= E_2 \otimes R_2 \otimes \cdots \otimes R_{k+1},\\
 &= (E_2 \otimes R_2 \otimes \eye[2] \otimes \cdots \otimes \eye[2])
    (E_2 \otimes \eye[2] \otimes R_3 \otimes \cdots \otimes \eye[2])
    \cdots
    (E_2 \otimes \eye[2] \otimes \eye[2] \otimes \cdots \otimes R_{k+1}),\\
 &= \prod_{i=1}^{k} E_2 \otimes \eeye[i-1] \otimes R_{i+1} \otimes \eeye[k-i].             
\end{align*}
Hence, the sum of Kronecker products \eqref{eq:lhsresult} is equal to the
product of sums, i.e.,
\begin{align*}
\eeye[k] \oplus \eOmega_k
% &= E_1 \otimes  \eye[2] \otimes \cdots \otimes \eye[2] \
% +\ E_2 \otimes R_2 \otimes \cdots \otimes R_{k+1},\\
 &= \prod_{i=1}^{k} E_1 \otimes \eeye[i-1] \otimes \eye[2] \otimes \eeye[k-i] \
 +\ \prod_{i=1}^{k} E_2 \otimes \eeye[i-1] \otimes R_{i+1} \otimes \eeye[k-i],\\
 &= \prod_{i=1}^{k} \Big[
      E_1 \otimes \eeye[i-1] \otimes \eye[2] \otimes \eeye[k-i] \ + \
      E_2 \otimes \eeye[i-1] \otimes R_{i+1} \otimes \eeye[k-i]\Big],
\end{align*}
because the mixed terms in the latter expression cancel out due to the property
$E_1 E_2 = 0$.
\end{proof}

The order of the terms in the Kronecker product expression \eqref{eq:dd}
can be changed because of the specific form of the $R_{i+1}$ matrices.

\begin{lemma}\label{lem:reorder}
The Kronecker product matrices in \eqref{eq:dd} satisfy
\begin{alignat}{5}
&E_1 &&\otimes \eeye[i-1] \otimes \eye[2] && \otimes \eeye[k-i] \ + \
      E_2 &&\otimes \eeye[i-1] \otimes R_{i+1} && \otimes \eeye[k-i] \label{eq:noreorder}\\
= \
&\eye[2] &&\otimes \eeye[i-1] \otimes E_1 && \otimes \eeye[k-i] \ + \
      R_{i+1} && \otimes \eeye[i-1] \otimes  E_2 && \otimes \eeye[k-i],
\label{eq:reorder}      
\end{alignat}      
for $i=1,\hdots,k$.
\end{lemma}
\begin{proof}
Splitting $\eye[2]$ and $R_{i+1}$ in their $E_1$ and $E_2$ components,
i.e.,
\begin{align*}
\eye[2] &= E_1 + \phantom{\omega_{2^{i+1}}} E_2,\\
R_{i+1} &= E_1 +          \omega_{2^{i+1}}  E_2,
\end{align*}
yields
\begin{align*}
E_1 \otimes \eeye[i-1] \otimes \eye[2] \ \otimes \eeye[k-i] \ + \
E_2 \otimes \eeye[i-1] \otimes R_{i+1} \otimes \eeye[k-i]
=&\,
      \begin{aligned}[t]
        &E_1 \otimes \eeye[i-1] \otimes E_1 \otimes \eeye[k-i] \ + \
         E_1 \otimes \eeye[i-1] \otimes \phantom{\omega_{2^{i+1}}} E_2
                                \otimes \eeye[k-i]\ + \\
        &E_2 \otimes \eeye[i-1] \otimes E_1 \otimes \eeye[k-i] \ + \
         E_2 \otimes \eeye[i-1] \otimes \omega_{2^{i+1}} E_2
                                \otimes \eeye[k-i],
      \end{aligned} \nonumber\\[5pt]
=&\,
      \begin{aligned}[t]
        &E_1 \otimes \eeye[i-1] \otimes E_1 \otimes \eeye[k-i] \hspace{-1em}
        && \ + \ \phantom{\omega_{2^{i+1}}} E_1 \otimes \eeye[i-1] \otimes E_2
                                           \otimes \eeye[k-i]\ + \\
        &E_2 \otimes \eeye[i-1] \otimes E_1 \otimes \eeye[k-i] \hspace{-1em}
        && \ + \ \omega_{2^{i+1}} E_2 \otimes \eeye[i-1] \otimes E_2
                                 \otimes \eeye[k-i],
      \end{aligned} \nonumber\\[5pt]
=&\,
      \begin{aligned}[t]
        &(E_1 + \phantom{\omega_{2^{i+1}}} E_2) \otimes
                \eeye[i-1] \otimes E_1 \otimes \eeye[k-i] \ + \\
        &(E_1 + \omega_{2^{i+1}} E_2) \otimes
                \eeye[i-1] \otimes E_2 \otimes \eeye[k-i],
      \end{aligned} \nonumber\\[5pt]
=&\,
  \begin{aligned}[t]
    &\eye[2] \otimes \eeye[i-1] \otimes E_1 \: \otimes \eeye[k-i] \ + \
     R_{i+1}  \otimes \eeye[i-1] \otimes  E_2 \otimes \eeye[k-i],
  \end{aligned}
\end{align*}
where we subsequently used the scalar shift and distributivity properties
\eqref{eq:kron-prop}, and combined the terms with $E_1$, respectively $E_2$,
in the $(i+1)$st position.
\end{proof}

Observe that \eqref{eq:dd} splits $\eeye[k] \oplus \eOmega_k$ into a product
of $k$ terms where each term has the form \eqref{eq:noreorder}.
\Cref{lem:reorder} allows us to modify each term in the product \eqref{eq:dd}
to an alternative representation~\eqref{eq:reorder}.
In what follows (\cref{sec:qgate}), we will see that every single term in
the product corresponds to a specific instance of an elementary operation on
a quantum computer and that the combination of
\Cref{thm:r2-diagdecomp,lem:reorder} leads to a new class of algorithms for
implementing $\eeye[k] \oplus \eOmega_k$.

% --- Connection with Quantum Computing --- %
\section{Quantum Fourier transform on a quantum computer}\label{sec:QC}
We now discuss how the decomposition shown in section~\ref{sec:qft}
enables the discrete Fourier transform to be computed efficiently on a quantum
computer.

% --- Rank-1 tensors --- %
\subsection{Quantum Fourier transform of rank-1 tensors on a classical computer}
\label{sec:rank1}
The Kronecker product form that appears in \eqref{eq:dd} allows for a
very efficient application to a rank-1 tensor of the form
\begin{equation}
\exx_n = x_1 \otimes x_2 \otimes \cdots \otimes x_n,
\label{eq:xrank1}
\end{equation}
where $x_i \inC[2] $ are vectors of length 2.
The multiplication of each factor in \eqref{eq:dd}
with \eqref{eq:xrank1} takes $\bigO(n)$ operations, whereas 
multiplying the diagonal matrix $\eeye[n] \oplus \eOmega_n$ with a vector of 
length $2^n$ takes $\bigO(2^n)$ operations.
Therefore, it is conceivable that the decomposition 
of the DFT matrix derived in the previous section may 
yield a more efficient algorithm for performing 
a discrete Fourier transform of a rank-1 tensor even 
on a classical computer. In this section, we show 
that this is not the case.  
However, in the next section, we will discuss a key
feature of a quantum computer that allows to take full
advantage of the Kronecker product structure of the 
decomposition given by \Cref{thm:r2-diagdecomp}, 
and yields an algorithm that only requires $\bigO(n^2)$ operations.

Recall from Theorem~\ref{thm:r2split} that $\eF_{n}^{\prime}$ can be 
decomposed as $\eF_{n}^{\prime} = (I_2 \otimes \eF_{n-1}^{\prime})(\eeye[n-1] \oplus \eOmega_{n-1})(H\otimes \eeye[n-1])$, where the factor
$\eeye[n-1] \oplus \eOmega_{n-1}$ can be further decomposed 
according to \eqref{eq:dd}.
The multiplication of the rightmost factor of $\eF_{n}^{\prime}$ with \eqref{eq:xrank1} yields
\[
(H\otimes \eeye[n-1]) \exx_n = (H x_1)\otimes x_2 \otimes \cdots \otimes x_n.
\]
Note that this multiplication only requires a constant number of
operations to produce $Hx_1$, which is independent of the problem size.
Furthermore, the result remains a rank-1 tensor. As long as we
keep the result in this rank-1 tensor product form, no additional
computation is required.

Next, we consider multiplying one factor of $\eeye[n-1] \oplus \eOmega_{n-1}$
derived in Theorem~\ref{thm:r2-diagdecomp} with a rank-1 tensor of the form \eqref{eq:xrank1}.
This amounts to a product of the matrix \eqref{eq:noreorder}, where $k=n-1$,
with the vector $\exx_n$, yielding
\begin{eqnarray}
\exy_n &=& \Big[
      E_1 \otimes \eeye[i-1] \otimes \eye[2] \otimes \eeye[n-i-1] \ + \
      E_2 \otimes \eeye[i-1] \otimes R_{i+1} \otimes \eeye[n-i-1]\Big] \exx_n \nonumber \\
&=& (E_1 x_1) \otimes x_2 \otimes \cdots \otimes x_n
+ (E_2 x_1) \otimes x_2 \otimes \cdots \otimes (R_{i+1} x_{i+1}) \otimes \cdots \otimes x_n. \label{eq:crx}
\end{eqnarray}
Thus, this multiplication also requires a constant number of operations.
However, the rank of the product is generally increased to 2,
unless one of the components of $x_1$ is zero.

If $\exy_n$ were to remain a rank-1 tensor after successive multiplications
of all factors of $\eeye[n-1] \oplus \eOmega_{n-1}$, the complexity of
multiplying the last two factors of $\eF_n^{\prime}$ with $\exx_n$ would be 
$\bigO(n)$.
Consequently, by applying this estimate recursively to $\eF_{k}^{\prime}$,
for $k = n-1,n-2,...,1$, the overall complexity of $\eF_n^{\prime}x$ calculation would be 
$\bigO(1+2+\cdots+n)=\bigO(n^2)$, even on a classical computer.

Unfortunately, as we can already see from \eqref{eq:crx}, that successive
applications of factors of the form 
     $ E_1 \otimes \eeye[i-1] \otimes \eye[2] \otimes \eeye[n-i] \ + \
       E_2 \otimes \eeye[i-1] \otimes R_{i+1} \otimes \eeye[n-i] $
for $i = 1,2,...,n$ tend to increase the rank of the tensor.
Nonetheless, as we saw in the proof of \Cref{thm:r2-diagdecomp}, the product
formulation \eqref{eq:dd} is a redundant representation where most terms 
cancel out.
If instead we decompose the diagonal matrix as
\begin{equation}
\eeye[n-1] \oplus \eOmega_{n-1} = E_1 \otimes \eye[2] \otimes \cdots \otimes \eye[2] + E_2 \otimes R_2 \otimes \cdots \otimes R_n,
\end{equation}
based on \eqref{eq:sumtokron} and \Cref{lemma:Omega}, then we get that the product of the diagonal
matrix with a rank-1 tensor is given by
\begin{equation}
(\eeye[n-1] \oplus \eOmega_{n-1}) \exx_n = (E_1 x_1) \otimes x_2 \otimes \cdots \otimes x_n + (E_2 x_1)\otimes (R_2 x_2) \otimes \cdots \otimes (R_n x_n).
\label{eq:drx}
\end{equation}
So we see that in general the rank increases to 2 after multiplication with
the diagonal matrix, just like after multiplication with a single term of the 
product representation \eqref{eq:crx}.

The redundant product representation \eqref{eq:dd} is more expensive to
evaluate classically than \eqref{eq:drx}, but will allow an efficient implementation on a quantum computer.

Although, in a classical algorithm, we can attempt to reduce the rank of each product before the 
multiplication with the next diagonal matrix is initiated, which is the approach
taken in ~\cite{Dolgov2012,Savostyanov2012} for a different tensor representation, 
it is, in general, not clear how much 
reduction can be achieved. Even if a rank reduction can be achieved,
it will not come free.  Consequently, the complexity of computing
the discrete Fourier transform on a classical computer through a sequence of
tensor products is higher than $\bigO(n^2)$ even if no explicit
sum is performed on the linear combination of rank-1 tensors, which requires 
a tremendous amount of storage for even a modest sized $n$, e.g., $n=100$.

\subsection{Qubits and quantum efficiency}

On a quantum computer, rank increase is not an issue.
A normalized vector of the corresponding tensor $\exx_n$ is 
kept as a quantum state, which can be viewed as a linear combination of 
a set of tensor product basis states, i.e.,
\begin{equation}
\exx_n =  \sum_{\ttj_1,\hdots,\ttj_n=1}^2 \alpha_{\ttj_1 \ttj_2 \hdots \ttj_n} \cdot \ej{1} \otimes \ej{2} \otimes \cdots \otimes \ej{n}.
\label{eq:ebasis}
\end{equation}
The $2^n$ coefficients $\alpha_{\ttj_1 \ttj_2 \hdots \ttj_n}$ of the expansion can be encoded and updated in $n$ quantum bits known 
as \textit{qubits}.
The basis vectors  $\ej{1} \otimes \ej{2} \otimes \cdots \otimes \ej{n}$ are the canonical basis of $\C[2^n]$.
For $n=1$, a normalized vector $\exx_1 \inC[2]$ can simply be 
written as $\exx_1 = \alpha e_1 + \beta e_2$ with $\alpha$ and $\beta$ satisfying
$|\alpha|^2 + |\beta|^2=1$. It follows from the superposition principle of quantum mechanics that a single \textit{qubit} can keep both the $\alpha$ and $\beta$ coefficients simultaneously.
For $n>1$, the $2^n$ coefficients satisfy the normalization condition
$ \sum_{\ttj_1,\hdots,\ttj_n=1}^2 |\alpha_{\ttj_{1} \ttj_{2} \hdots \ttj_{n}}|^2 = 1$ 
and they can be stored in the state of just $n$ qubits.
Observe that an $n$-qubit state \eqref{eq:ebasis} can be in a complete superposition, i.e., having all
coefficients $\alpha_{\ttj_1 \ttj_2 \hdots \ttj_n}$ different from zero, and still allow a data sparse
representation up to some finite precision in $\bigO(n)$ space in memory of a classical computer if it 
is a rank-1 tensor \eqref{eq:xrank1}.

An additional feature of a quantum computer is that it can store a tensor $\exx_n$ which has
tensor rank $r>1$ in only $n$ qubits, such a state is also known as an entangled state in
quantum mechanics.
The minimal representation of an entangled state,
\begin{equation}
\exx_n = \sum_{i=1}^r \alpha_{i} \cdot x\p{i}_1 \otimes x\p{i}_2 \otimes \cdots \otimes x\p{i}_n,
\label{eq:xtensorrank}
\end{equation}
where $x\p{i}_k \inC[2]$, $\| x\p{i}_k\|_2 = 1$ for $i=1,2,\hdots,r$,
$k = 1, 2, \hdots, n$,
$\sum_{i=1}^r |\alpha_i|^2 = 1$, and $r$ is an integer rank,
is a canonical polyadic decomposition of the $n$-way tensor $\exx_n$,
see~\Cref{def:v-rank1}.
Notice that if \eqref{eq:xtensorrank} is not a \emph{minimal} representation for $\exx_n$, then the
actual tensor rank can be strictly smaller than $r$.
Storing a tensor \eqref{eq:xtensorrank} up to finite precision requires $\bigO(nr)$ classical memory.
For $n = 2$, every tensor $\exx_2$ can be represented as a sum of two terms,
\begin{equation}
\exx_2 = \alpha_1 \cdot x\p1_1 \otimes x\p1_2 + \alpha_2 \cdot x\p2_1 \otimes x\p2_2,
\label{eq:svd}
\end{equation}
which is a singular value decomposition of a $2 \times 2$ matrix. 

How a quantum computer does this is beyond the scope of this paper, and not relevant for 
our discussion. 
It is also worth pointing out that, although a quantum
computer can store a quantum state \eqref{eq:xtensorrank}, it
cannot be accessed in the same way a register or memory is accessed on 
a classical computer. We will discuss the implication of this feature in terms
of how QFT can be used later.

The multiplication of a matrix operator 
with a vector $\exx_n$ of tensor rank $r$ \eqref{eq:xtensorrank} can be carried out by first multiplying the matrix with 
the tensor factors $x\p{i}_1 \otimes x\p{i}_2 \otimes \cdots \otimes x\p{i}_n$, $i=1,\hdots,r$.
When the matrix operation can be written in 
terms of Kronecker products
of $2\times 2$ matrices, the multiplication does not require a vector to be 
explicitly formed.  All we have to do is to multiply each $2\times 2$ matrix
in a Kronecker product with an appropriate component of each of tensor factors $x\p{i}_j$.
For example, applying the matrix in \eqref{eq:noreorder}, with $k=n-1$, to 
a tensor factor in \eqref{eq:xtensorrank} yields
\begin{equation}
(E_1 x\p{i}_1)  \otimes x\p{i}_2 \otimes \cdots \otimes x\p{i}_j \otimes x\p{i}_{j+1} \otimes x\p{i}_{j+2} \cdots \otimes x\p{i}_n 
\ + \ (E_2  x\p{i}_1) \otimes x\p{i}_2 \otimes \cdots \otimes x\p{i}_j \otimes (R_{j+1} x\p{i}_{j+1}) \otimes x\p{i}_{j+2} \otimes \cdots \otimes x\p{i}_n, 
\label{eq:cR}
\end{equation}
similar to the rank-1 case in \eqref{eq:crx}.
As the tensor rank representation \eqref{eq:xtensorrank} has a minimal number of terms by
definition, the multiplication of the matrix \eqref{eq:reorder} with $\exx_n$ requires classically at
least $r$ of the evaluations shown above for $i= 1,\hdots, r$. In essence this amounts to computing
$R_{j+1} x\p{i}_1$ for every tensor factor $x\p{i}_1$. 
These changes are managed by a quantum computer
instantaneously at a cost that is independent of $r$, even as the matrix operator would increase the rank $r$. 
No additional computation is required to collect or 
recompute the factors.

Because each factor in the matrix decomposition derived in Theorem~\ref{thm:r2-diagdecomp}
contains a sum of Kronecker products of $2\times 2$ identity matrices
with at most one non-identity matrix. The only operations we
need to account for are the multiplications of $R_i$'s and $H$ with 
tensor factors $x\p{i}_j$ and the total cost for updating the state tensor
\eqref{eq:xtensorrank} is the same for $r=1$ as it is for a full rank tensor.
This particular feature of a quantum computer 
results in quantum efficiency. It also means that the complexity 
of multiplying \eqref{eq:dd} with a vector of size $2^n$ 
can be evaluated in terms of the number of $R_i$'s in all of the factors, 
regardless how large the 
rank of the vector is when viewed as a $n$-dimensional tensor. Because
the total number of $R_i$'s is $\bigO(n^2)$, the QFT can achieve 
$\bigO(n^2)$ complexity overall.

\subsection{Quantum gates and quantum circuits} \label{sec:qgate}

Writing down and analyzing the QFT in terms of products of matrices in
the form of \eqref{eq:dd} can be 
cumbersome. In the quantum computing literature, there is a convenient 
graphical way to depict the unitary transformations on a single or an $n$-qubit system.
The building 
blocks of the transformation are $2 \times 2$ unitary matrices such as
the matrix $H$ in \eqref{er2fft2} or $R_{i+1}$ in \eqref{eq:dd}.
Each one of them is drawn
as a square box labelled by a letter and referred to as a quantum gate.
A single quantum gate is applied to a single qubit,
which is drawn as a line to the left of the gate.  
The result $\phi=H\psi$ is also drawn as a line to the right of the 
gate.  In quantum mechanics, $\psi$ and $\phi$ are often
written in Dirac's ket notation as shown in the diagram below
\[\vcenter{
\Qcircuit @C=1em @R=1em {
\lstick{\ket{\psi}} & \gate{H} & \qw & \ \ket{\phi}.
}}
\]

Successive applications of a sequence of $m$ $2\times 2$ unitaries
$U_1,U_2,\ldots,U_m$ to a
single qubit quantum state $\psi$, i.e., 
$\phi = U_m \cdots U_2 U_1 \psi$, can be drawn as
\[\vcenter{
\Qcircuit @C=1em @R=1em {
\lstick{\ket{\psi}} & \gate{U_1} & \gate{U_2} & \qw & \cdots & & \gate{U_m} &\qw & \ \ket{\phi}.
}}
\]
Such a diagram is often referred to as a quantum circuit~\cite{Deutsch1989}. Notice 
the first unitary to be applied, $U_1$, is the leftmost quantum gate in the
quantum circuit.

The application of a unitary in the form of $U \otimes I_2$ to 
a two qubit state $\exx_2$ can be drawn as 
\[\vcenter{
\Qcircuit @C=1em @R=1em {
\lstick{q_1} & \gate{U} & \qw  \\
\lstick{q_2} & \qw      & \qw 
}}
\]
The $q_1$ and $q_2$ symbols are used to label the qubits.
If the input to this circuit is a rank-1 tensor $\exx_2 = x_1 \otimes x_2$,
the output of the circuit is simply $ Ux_1 \otimes x_2$, and a classical simulation
of the circuit has the same cost as its executing on a quantum computer. 
Notice that the $2\times 2$ identity matrix does not require a quantum gate.
In general, if $\exx_2$ is a two qubit state of maximal rank 2, cfr. \eqref{eq:svd},
the circuit performs the computation
\[
\exy_2 = \alpha_1 \cdot (U x\p1_1) \otimes x\p1_2 + \alpha_2 \cdot (U x\p2_1) \otimes x\p2_2,
\]
which requires two products with $U$ to compute classically.
It follows that if an $n$-qubit state has rank $\bigO(2^n)$, the quantum efficiency of
quantum computers can lead to exponential speedups in comparison to classical computers.

This type of diagram can be easily generalized for unitary transformations
applied to state vectors encoded by several qubits.
In the context of the QFT, a particularly interesting unitary transformation
is the one represented by the matrices of \Cref{lem:reorder}. 
As we saw in \eqref{eq:cR}, the action on the $(j+1)$st tensor factor
$x\p{i}_{j+1}$ depends on the coefficients of the first tensor
factor, $E_1 x\p{i}_1$ and $E_2 x\p{i}_1$.
If the first tensor factor is equal to $e_1$, then the identity
matrix is applied to $x\p{i}_{j+1}$.
If $x\p{i}_1$ is equal to $e_2$, then the $(j+1)$st factor changes
to $R_{j+1} x\p{i}_{j+1}$.
Consequently, this type of transformation is called
a \textit{controlled} unitary with the first qubit as the controlling qubit and the $(j+1)$st qubit as the target qubit. 
In general, $x\p{i}_1$ can be in a superposition of $e_1$ and $e_2$ and the result is the
linear combination \eqref{eq:cR} which can increase the tensor rank as we discussed in \Cref{sec:rank1}.
The diagrammatic notation for this
controlled unitary is drawn on the left in the diagram below. Note that a solid circle is
drawn on the line representing the controlling qubit. 
It is connected to the qubit to be transformed via a vertical line.
\Cref{lem:reorder} shows the equivalence between the controlled unitaries on
the left and the right for controlled-$R$ gates.
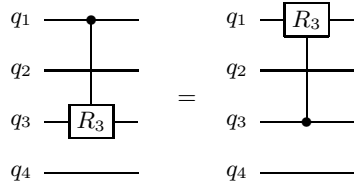
\begin{equation*} \vcenter{
\Qcircuit @C=1em @R=1.4em {
   \lstick{q_1} & \ctrl{2}       &  \qw  &&&&& \lstick{q_1} &  \gate{R_3} &  \qw\\
   \lstick{q_2} &  \qw           &  \qw  &&  \raisebox{-2.5em}{=} &&&    \lstick{q_2} &  \qw           &  \qw\\
   \lstick{q_3} & \gate{R_3} & \qw   &&&&& \lstick{q_3} & \ctrl{-2} & \qw \\
   \lstick{q_4} &  \qw           & \qw   &&&&&    \lstick{q_4} &  \qw      & \qw}
}\end{equation*}
\vspace{0.5em} %otherwise the next paragraph is really close

Another useful controlled unitary is the \emph{controlled}-NOT or CNOT
 unitary. As a 2-qubit operator with the first qubit being the control and the second
 qubit being the target, it can be written as
\[
X_c =  E_1 \otimes \eye[2] + E_2 \otimes X,
\]
where $X$ is the NOT operator
\[
X = \begin{bmatrix} 0 & 1\\ 1 & 0 \end{bmatrix},
\]
that maps $e_1$ to $e_2$ and vice versa.
The conventional diagram for a CNOT gate in quantum computing is
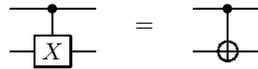
\[
\Qcircuit @C =1em @R=1em{
& \ctrl{1} & \qw & & & & & \ctrl{1} & \qw\\
& \gate{X} & \qw & & \raisebox{2em}{=} & & & \targ & \qw
}
\raisebox{-1em}
{
}
\]
One particular use of the CNOT gate is to implement the SWAP operator
defined as
\[
\text{SWAP}( \psi_1 \otimes \psi_2) = \psi_2 \otimes \psi_1.
\]
This is sometimes represented by 
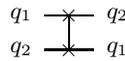
\[
\Qcircuit @C=1em @R=1.4em {
\lstick{q_1} & \qswap\qwx[1] & \qw & q_2\\
\lstick{q_2} & \qswap        & \qw & q_1
}
\raisebox{-2em}
{
}
\]
in a quantum circuit.  
It can be easily verified that a SWAP gate can be decomposed as the product
of three CNOTS with alternating controlling qubits. As a result,  
the CNOT implementation of a SWAP gate can be drawn as
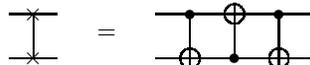
\[
\Qcircuit @C =1em @R=1em{
& \qswap\qwx[1]  & \qw & & & & & \ctrl{1} & \targ & \ctrl{1} & \qw\\
& \qswap & \qw & & \raisebox{2em}{=} & & & \targ & \ctrl{-1} & \targ & \qw
}
\]

Controlled unitaries can also be defined to have multiple target qubits.
The matrix $E_1 \otimes \eeye[k] + E_2 \otimes \eU_k$ is a controlled-$\eU_k$
gate with the first qubit as a control and qubits $2, \hdots, k+1$ as target.
It follows from \eqref{eq:bitrev} that a bit reversal operator $\eP_n$ can 
be written in terms of a sequence of SWAP operations. Therefore, 
$\eP_n$ can be implemented using sequence of CNOT gates.

In general, a unitary transformation $\eU_n$ applied to an 
$n$-qubit state $\expsi_n$
can be drawn as 
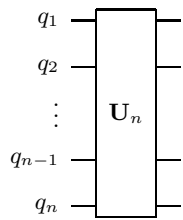
\[
\Qcircuit @C=1em @R=1em {
\lstick{q_1} & \multigate{4}{\eU_n} & \qw \\
\lstick{q_2} & \ghost{\eU_n} & \qw \\
\lstick{\raisebox{.5em}{\vdots}} & \pureghost{\eU_n} &\\
\lstick{q_{n-1}} & \ghost{\eU_n} & \qw \\
\lstick{q_{n}}  & \ghost{\eU_n} & \qw 
}
\]
An alternative and simplified way to draw this transformation is
\[ \Qcircuit @C=1em @R=1em {
 \lstick{\ket{\expsi_n}} & \qw {/} & \gate{\eU_n} & \qw & \quad \ket{\exphi_n},
}
\]
where the line with a `\,/\,' through indicates a wire representing $n$ qubits.

\subsection{Quantum circuit for the quantum Fourier transform}\label{qubit}

We now show how the QFT matrix decomposition derived in \Cref{thm:r2fft} and \Cref{thm:r2-diagdecomp} can be expressed succinctly by using the 
quantum circuit diagrams introduced in \cref{sec:qgate}.
We start from \Cref{thm:r2fft} and write down the circuit
representation of the decomposition of the DFT matrix $\eF_n \inCC{2^n}$ 
as follows
\begin{equation}\label{cirq:qft1}
\vcenter{
\Qcircuit @C=1em @R=1em {
& & & & & & & & & \lstick{q_1} & \qw & \multigate{4}{\eB_{n}} & \qw   & \qw  & \cdots & & \qw & \qw &\multigate{4}{\eP_n} & \qw\\
& & & & & & & & & \lstick{q_2} & \qw & \ghost{\eB_n} & \multigate{3}{\eB_{n-1}} & \qw  & \cdots & & \qw & \qw & \ghost{\eP_n} & \qw\\
& \lstick{\ket{\expsi_n}} & \qw{/} & \gate{\eF_{n}} & \qw  & = & & & \raisebox{.5em}{\vdots} & & & \pureghost{\eB_n} & & & \raisebox{.5em}{\vdots} & & & & & \\
& & & & & & & & & \lstick{q_{n-1}} & \qw & \ghost{\eB_n} & \ghost{\eB_{n-1}} & \qw & \cdots &  & \multigate{1}{\eB_2} & \qw & \ghost{\eP_n} & \qw\\
& & & & & & & & & \lstick{q_n} & \qw & \ghost{\eB_n} & \ghost{\eB_{n-1}} & \qw & \cdots &  & \ghost{\eB_2} & \gate{\eB_1} & \ghost{\eP_n} & \qw
}}
\end{equation}
Note that the gate associated with $\eB_{k+1}$ is applied to the last $k+1$ qubits.

Each multi-qubit gate block of this quantum circuit can be decomposed into a sequence of single or two-qubit gates.
For example, the bit-reversal permutation matrix $\eP_n$ can first be
written in terms of $\lfloor n/2 \rfloor$ SWAP gates as follows
\begin{equation}\label{cirq:br}
\vcenter{
\Qcircuit @C = 1em @R = 1.4em {
& & & & & & & & & & \lstick{q_1} & \qswap & \qw & \qw & \cdots \\
& & & & & & & & & & \lstick{q_2} & \qw & \qswap & \qw & \cdots \\
& \lstick{\ket{\expsi_n}} & \qw{/} & \gate{\eP_{n}} & \qw  & = & & & & \raisebox{.5em}{\vdots} & & & & & \raisebox{.5em}{\vdots} \\
& & & & & & & & & & \lstick{q_{n-1}} & \qw & \qswap \qwx[-2] & \qw & \cdots\\
& & & & & & & & & & \lstick{q_n} & \qswap \qwx[-4] & \qw & \qw & \cdots
}}
\end{equation}
Since each SWAP can be implemented with 3 CNOT gates, $\eP_n$ requires $\lfloor 3n/2 \rfloor$ CNOT gates.

It follows from \eqref{er2fft2} that the $\eB_{k+1}$ gates can be decomposed 
further as
\begin{equation}\label{cirq:Bi}
\vcenter{
\Qcircuit @C = 1em @R = 1em {
& & & & & & & & & \lstick{q_1} & \gate{H} & \multigate{4}{\eeye[k] \oplus \eOmega_k} & \qw\\
& & & & & & & & & \lstick{q_2} & \qw & \ghost{\eeye[k] \oplus \eOmega_k}  & \qw\\
&  \lstick{\ket{\expsi_{k+1}}} & \qw{/} & \gate{\eB_{k+1}} & \qw  & = & &  &\raisebox{.5em}{\vdots}  & & & & \\
& & & & & & & & & \lstick{q_k} & \qw & \ghost{\eeye[k] \oplus \eOmega_k}  & \qw\\
& & & & & & & & & \lstick{q_{k+1}} & \qw & \ghost{\eeye[k] \oplus \eOmega_k}  & \qw\\
}}
\end{equation}

Furthermore, the decomposition given in \Cref{thm:r2-diagdecomp} can be used to 
rewrite the $\eeye[k] \oplus \eOmega_k$ block in
terms of controlled operations involving two qubits, as shown in the rightmost circuit below
\begin{equation}\label{cirq:D}
\vcenter{
\Qcircuit @C=1em @R=1em {
\lstick{q_1}  & \multigate{4}{\eeye[k] \oplus \eOmega_{k}} & \qw & & & & & \lstick{q_1} & \ctrl{1} & \qw &  & & & &  \lstick{q_1} & \control \ar @{-} +<0em,-1.05em> \qw & \qw & & & & & \lstick{q_1} & \ctrl{4} & \ctrl{3} & \qw & \cdots & & \ctrl{1} & \qw\\
\lstick{q_2} & \ghost{\eeye[k] \oplus \eOmega_{k}} & \qw& & & & & \lstick{q_2} & \multigate{3}{\eOmega_{k}} & \qw  & & & & & \lstick{q_2} & \gate{R_2} & \qw & & & & & \lstick{q_2} & \qw & \qw & \qw & \cdots & & \gate{R_2} & \qw \\
\lstick{\raisebox{.5em}{\vdots}} & \pureghost{\eye[k] \oplus \Omega_{k}} & & & = & & \raisebox{.5em}{\vdots}  & & & & &= & & & \lstick{\raisebox{.5em}{\vdots}} & \raisebox{.5em}{\vdots}  & & & = & & \raisebox{.5em}{\vdots} & & & & & \raisebox{.5em}{\vdots} & \\
\lstick{q_k} & \ghost{\eeye[k] \oplus \eOmega_{k}} & \qw & & & & &\lstick{q_k} &  \ghost{\eOmega_{k}} & \qw &  & & & & \lstick{q_k} & \gate{R_k} & \qw & & & & & \lstick{q_k} & \qw & \gate{R_k} & \qw & \cdots & & \qw & \qw\\
\lstick{q_{k+1}} & \ghost{\eeye[k] \oplus \eOmega_{k}} & \qw & & & & &\lstick{q_{k+1}} &  \ghost{\eOmega_{k}} & \qw &  & & & & \lstick{q_{k+1}} & \gate{R_{k+1}} & \qw & & & & & \lstick{q_{k+1}} & \gate{R_{k+1}} &\qw  & \qw & \cdots & & \qw & \qw
\gategroup{2}{16}{5}{16}{.7em}{-}
}}\hspace{-15pt}
\end{equation}
The circuit on the left is a dense $k+1$ qubit diagonal gate $\eeye[k] \oplus \eOmega_k$, 
which we can rewrite as a controlled-$\eOmega_k$
with the first qubit as control and the next $k$ qubits as target based on \eqref{eq:lhsresult}.
From \Cref{lemma:Omega} we have that $\eOmega_k$ can be decomposed as a Kronecker product of $R$ matrices,
which gives the third circuit above, where the control qubit controls the whole group $R_2 \otimes \cdots \otimes R_{k+1}$.
This circuit can evidently be expanded as a product of separate controlled-$R$ gates to obtain the rightmost circuit,
which gives a compact visual proof of \Cref{thm:r2-diagdecomp} only using circuit diagrams.
The decomposition of $\eeye[k] \oplus \eOmega_k$ in $k$ gates is exceptionally efficient, because the
implementation of an $n$-qubit
diagonal operator requires $\bigO(2^n)$ gates in general~\cite{Welch2015,Bullock2004}.

All the controlled-$R$ gates in this decomposition are diagonal matrices and commute so the order
in which they are applied does not change the outcome.
Furthermore, we have from \Cref{lem:reorder} that the roles of the control and target qubits can
be swapped without changing the result.
This leads us to a whole class of equivalent quantum circuits for $\eeye[k] \oplus \eOmega_k$.

Finally combining \eqref{cirq:qft1}--\eqref{cirq:D} gives the following complete QFT quantum circuit
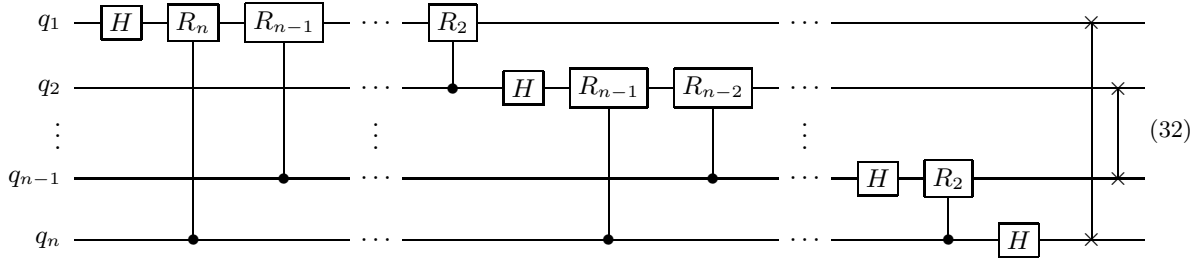
\begin{equation}
\resizebox{.8\hsize}{!}{$
\vcenter{
\Qcircuit @C=1em @R=1em {
\lstick{q_1} & \gate{H} & \gate{R_{n}} & \gate{R_{n-1}} & \qw & \cdots & & \gate{R_2} & \qw & \qw & \qw &  \qw & \cdots & & \qw & \qw & \qw & \qw & \qswap & \qw & \qw \\
\lstick{q_2}  & \qw & \qw & \qw  & \qw  & \cdots & & \ctrl{-1} & \gate{H} & \gate{R_{n-1}} & \gate{R_{n-2}} & \qw & \cdots & & \qw & \qw & \qw & \qw & \qw & \qswap & \qw \\
\lstick{\raisebox{.5em}{\vdots}} & & & & & \raisebox{.5em}{\vdots} & & & & & & &  \raisebox{.5em}{\vdots}  \\
\lstick{q_{n-1}}  & \qw & \qw & \ctrl{-3} & \qw  & \cdots & & \qw & \qw & \qw & \ctrl{-2} & \qw & \cdots & &\gate{H} & \gate{R_2} & \qw & \qw & \qw & \qswap \qwx[-2] & \qw \\
\lstick{q_n}  & \qw & \ctrl{-4} & \qw & \qw & \cdots & & \qw & \qw & \ctrl{-3} & \qw & \qw & \cdots & & \qw & \ctrl{-1} & \gate{H} & \qw & \qswap \qwx[-4] & \qw & \qw
}}
$}\hspace{-15pt}
\end{equation}
In this QFT circuit, we swapped the control and target qubits compared to \Cref{thm:r2-diagdecomp} for all controlled-$R$ gates.
This particular choice gives the same QFT circuit that is discussed in \cite[section 5.1]{NC2010}.

The gate count of for each type of gate used in a QFT is summarized 
in \Cref{tab}.
The total number of elementary gates required to implement the QFT is $\bigO(n^2)$, which defines the complexity of the QFT algorithm. It represents an exponential improvement over the $\bigO(2^n n)$ complexity of the FFT algorithm.

\begin{table}[ht]
\centering
\caption{Gate count for the QFT circuit on $n$ qubits.}
\begin{tabular}[t]{lcc}
\toprule
Matrix & Count & Gate type\\
\midrule
$\eP_n$ & $\lfloor 3n/2 \rfloor$ & CNOT \\
\multirow{ 2}{*}{$\eA_{n}^{(0)} \cdots \eA_{n}^{(n-1)}$ } & $n$ & Hadamard\\
 & $n(n-1)/2$ & controlled-$R$\\
\bottomrule
\end{tabular}\label{tab}
\end{table}%

\section{Generalization to Radix-$d$ Quantum Fourier transform}\label{sec:qudit}

The QFT decomposition presented in section~\ref{sec:qft} can be easily 
generalized for DFT matrices of dimension $d^n \times d^n$ for any integer $d>2$.
The generalization relies on a radix-$d$ FFT decomposition and
the base-$d$ representation of an integer $j$ 
\[
j = [\ttj_1\ttj_2\cdots\ttj_{n-1}\ttj_n] = \ttj_1 \cdot d^{n-1} + \ttj_2 \cdot d^{n-2} + \cdots + \ttj_{n-1} \cdot d^1 + \ttj_n \cdot d^0,
\]
where $\ttj_i = \lbrace 0, 1, \hdots, d-1 \rbrace$ for $i=1,\hdots, n$.
Hence, \Cref{lem:bin-pow,lem:bin-vec} generalize in a straightforward manner
to base-$d$.

% --- Radix-d decomposition --- %
\subsection{Radix-$d$ decomposition of the discrete Fourier transform matrix}\label{rd}

In this section, we use $\eA_n = A_{d^n}$ to denote matrices of dimension $d^n$.
%%% DEFINITION %%%
\begin{definition}\label{def:bOnbd}
Define $\eOmega_n \inC[d^n][d^n], R_n \inC[d][d]$ as the following matrices:
\[
\eOmega_n = \Omega_{d^n} :=
\begin{bmatrix}
\omega_{d^{n+1}}^0 & & & \\
 & \omega_{d^{n+1}}^1 & & \\
 & & \ddots & \\
 & & & \omega_{d^{n+1}}^{d^n-1}
\end{bmatrix},
\qquad
R_n = 
\begin{bmatrix}
\omega_{d^n}^0 & & & \\
& \omega_{d^n}^1 & &\\
& & \ddots & \\
& & & \omega_{d^n}^{d-1}
\end{bmatrix},
\]
where $\omega_{d^i} = \exp(\frac{-2\pi\I}{d^i})$.
\end{definition}

If the DFT matrix is of dimension $d^n \times d^n$, then it can be factored
in $d \times d$ block partitioning, also known as the radix-$d$ factorization 
of the DFT matrix \cite{Sloate1974,Pease1968}.
This is summarized in the following two theorems which naturally extend the results
of the radix-$2$ framework of \cref{r2}.

%%% THEOREM %%%
\begin{theorem}[See \cite{Sloate1974}]\label{trdsplit}
Let $\eF_{n}^{\prime} = \eP_n \eF_n$ be the $d^n \times d^n$ base-$d$ reversed DFT
matrix. Then $\eF_{n}^{\prime}$ admits the following factorization:
\begin{equation}\label{eq:rdsplit}
\eF_{n}^{\prime} = (\eye[d] \otimes \eF_{n-1}^{\prime})(\eeye[n-1] \oplus \eOmega_{n-1} \oplus \cdots \oplus \eOmega_{n-1}^{d-1})(F_{d} \otimes \eeye[n-1]),
\end{equation}
where $\eOmega_{n-1}$ is given by \Cref{def:bOnbd} and $F_d$ is the $d \times d$ DFT matrix defined in \Cref{def:dft-matrix}.
\end{theorem}

\Cref{eq:rdsplit} essentially decomposes $\eF_n^{\prime}$ into the Kronecker product of the $d^{n-1} \times d^{n-1}$ 
DFT matrix $\eF_{n-1}^{\prime}$ and a $d \times d$
block matrix with block size $d^{n-1} \times d^{n-1}$.
The full radix-$d$ factorization of the DFT matrix follows from \Cref{trdsplit}
via the same simple induction argument used in the proof of \Cref{thm:r2fft}.

%%% THEOREM %%%
\begin{theorem}\label{trdfft}
The $d^n \times d^n$ DFT matrix $\eF_n$ can be factored as:
\begin{equation}\label{erdfft}
\eF_{n} = \eP_n \eF_{n}^{\prime} = \eP_n \eA^{(0)}_n \hdots \eA^{(n-1)}_{n},
\end{equation}
where, for $k=0, \hdots, n-1$,
\begin{equation}\label{erdfft2}
\eA^{(k)}_{n} = \eeye[n-k-1] \, \otimes \eB_{k+1},
\qquad \text{and} \qquad
\eB_{k+1}
= (\eeye[k] \oplus \eOmega_k \oplus \cdots \oplus \eOmega_{k}^{d-1})(\eF_1 \otimes \eeye[k]).
\end{equation}
with $\eOmega_{k}$ defined in \Cref{def:bOnbd} and $\eF_1 = F_d$ the $d \times d$ DFT matrix given by \Cref{def:dft-matrix}.
\end{theorem}

\begin{definition}\label{def:dPerm}
The base-$d$ reversal permutation matrix $\eP_n \inC[d^n][d^n]$
is defined as the permutation matrix which satisfies:
\[\eP_n (v_1 \otimes \cdots \otimes v_n) = v_n \otimes \cdots \otimes v_1,\]
where $v_k \inC[d]$, for $k=1,\hdots,n$.
\end{definition}

\Cref{def:dPerm} shows that the permutation matrix can be implemented
in $\lfloor n/2 \rfloor$ SWAPs of $d$ dimensional vectors in 
a Kronecker product of $n$ such vectors.

% --- Radix-d / qudit case --- %
\subsection{Decomposition of the diagonal matrix}

In this section we decompose the diagonal matrix
$\eeye[k] \oplus \eOmega_k \oplus \cdots \oplus \eOmega_{k}^{d-1}$
into $k$ controlled-$R$ unitaries, with $R$ defined in \Cref{def:bOnbd},
following an analogous strategy to \cref{sec:qft}.
We start again with a decomposition of $\eOmega_n$
as a Kronecker product of $d \times d$ $R_i$ matrices.

%%% LEMMA %%%
\begin{lemma}
\label{lem:Rnbd}
Let $R_n$ be defined by \Cref{def:bOnbd}. Then
\[
R_n^{d^j} = R_{n-j},
\]
for $j \inN$.
\end{lemma}
\begin{proof}
The proof proceeds analogously to \Cref{lem:Rn}.
\end{proof}
%%% LEMMA %%%
\begin{lemma}\label{lemma:Omegabd}
Let $\eOmega_n$ and $R_n$ be defined by \Cref{def:bOnbd}.
Then
\[
\eOmega_n = R_2 \otimes R_3 \otimes \cdots \otimes R_n \otimes R_{n+1}.
\]
\end{lemma}
\begin{proof}
The proof proceeds analogously to \Cref{lemma:Omega}, replacing 
the binary representation with base-$d$ representation and using \Cref{lem:Rnbd}.
\end{proof}

The generalization of \Cref{thm:r2-diagdecomp} to the decomposition of 
the matrix
$\eeye[k] \oplus \eOmega_k \oplus \cdots \oplus \eOmega_{k}^{d-1}$ can 
be stated as follows.

%%% THEOREM %%%
\begin{theorem}\label{thm:rd-diagdecomp}
The diagonal operators $\eeye[k] \oplus \eOmega_k \oplus \cdots \oplus \eOmega_{k}^{d-1} \inCC{d^{k+1}}$ from \Cref{trdfft}
admit the decomposition:
\begin{equation}\label{eq:dDDecomp}
\eeye[k] \oplus \eOmega_k \oplus \cdots \oplus \eOmega_{k}^{d-1} \ = \
\prod_{i=1}^{k}
\sum_{\ell=1}^{d}
E_\ell \otimes \eeye[i-1] \otimes R_{i+1}^{\ell-1} \otimes \eeye[k-i]
\end{equation}
with $E_\ell := e_\ell e_\ell\T \inC[d][d]$ and $R_{i+1}$ defined in \Cref{def:bOnbd}.
\end{theorem}
%%% PROOF %%%
\begin{proof}
We start by splitting the direct sum of $d$ diagonal matrices as a sum of Kronecker
products, which yields
\begin{equation}\label{eq:dsumtokron}
\eeye[k] \oplus \eOmega_k \oplus \cdots \oplus \eOmega_{k}^{d-1}
 = E_1 \otimes \eeye[k] + E_2 \otimes \eOmega_k + \cdots + E_d \otimes \eOmega_{k}^{d-1},
\end{equation}
where we used a generalization of \eqref{eq:sumtokron}.
Next, we rewrite the $\ell$th term on the right-hand side in a redundant form based on the mixed
product identity of Kronecker product \eqref{eq:kron-mixedprod}.
Using \Cref{lemma:Omegabd} and the property $E_\ell = E_\ell^2$, we get
\begin{align*}
E_\ell \otimes \eOmega_k^{\ell-1}
 &= E_\ell \otimes  R_2^{\ell-1} \otimes \cdots \otimes R_{k+1}^{\ell-1},\\
 &= (E_\ell \otimes R_2^{\ell-1} \otimes \eye[2] \otimes \cdots \otimes \eye[2])
    (E_\ell \otimes \eye[2] \otimes R_3^{\ell-1} \otimes \cdots \otimes \eye[2])
    \cdots
    (E_\ell \otimes \eye[2] \otimes \eye[2] \otimes \cdots \otimes R_{k+1}^{\ell-1}),\\
 &= \prod_{i=1}^{k} E_\ell \otimes \eeye[i-1] \otimes R_{i+1}^{\ell-1} \otimes \eeye[k-i].
\end{align*}
Hence, the sum of Kronecker products \eqref{eq:dsumtokron} is equal
to the product of sums, i.e.,
\begin{equation*}
\eeye[k] \oplus \eOmega_k \oplus \cdots \oplus \eOmega_{k}^{d-1} 
=
\sum_{\ell=1}^{d}
\prod_{i=1}^{k}
E_\ell \otimes \eeye[i-1] \otimes R_{i+1}^{\ell-1} \otimes \eeye[k-i]
=
\prod_{i=1}^{k}
\sum_{\ell=1}^{d}
E_\ell \otimes \eeye[i-1] \otimes R_{i+1}^{\ell-1} \otimes \eeye[k-i]
\end{equation*}
because the mixed terms in the latter expression cancel out due to the
property $E_i E_j = 0$ for $i \neq j$.
\end{proof}

The position of the $R_{i+1}$ and $E_\ell$ matrices in the Kronecker
product expression can again be swapped.

%%% LEMMA %%%
\begin{lemma}\label{lem:dreorder}
The Kronecker product matrices in \eqref{eq:dDDecomp} satisfy
\begin{equation}
\sum_{\ell=1}^{d}
E_\ell \otimes \eeye[i-1] \otimes R_{i+1}^{\ell-1} \otimes \eeye[k-i]
\ = \
\sum_{\ell=1}^{d}
R_{i+1}^{\ell-1} \otimes \eeye[i-1] \otimes E_\ell  \otimes \eeye[k-i],
\end{equation}
for $i=1,\hdots,k$.
\end{lemma}
\begin{proof}
Splitting $R_{i+1}$ in its $E_1, E_2, \hdots, E_d$ components, i.e.,
\begin{equation*}
R_{i+1} = E_1 + \omega_{d^{i+1}} E_2 + \hdots + \omega_{d^{i+1}}^{d-1} E_d,
\end{equation*}
yields
\begingroup
\allowdisplaybreaks
\begin{align*}
\sum_{\ell=1}^{d}
E_\ell & \otimes \eeye[i-1] \otimes R_{i+1}^{\ell-1} \otimes \eeye[k-i]\\[5pt]
\ = \  &\begin{aligned}[t]
        &E_1 \otimes \eeye[i-1] \otimes (E_1 + \phantom{\omega_{d^{i+1}}} E_2 + \cdots + \phantom{\omega_{d^{i+1}}^{(d-1)^2}} E_d)
                           \otimes \eeye[k-i] \ + \\
        &E_2 \otimes \eeye[i-1] \otimes (E_1 + \omega_{d^{i+1}} E_2 + \cdots + \omega_{d^{i+1}}^{d-1\phantom{()^2}} E_d)
                           \otimes \eeye[k-i]\ + \\
        & \cdots\ + \\
        &E_d \otimes \eeye[i-1] \otimes (E_1 + \omega_{d^{i+1}}^{d-1} E_2 + \cdots + \omega_{d^{i+1}}^{(d-1)^2} E_d)
                           \otimes \eeye[k-i],
      \end{aligned}\\[5pt]           
\ = \ & \begin{aligned}[t]
        &E_1 \otimes \eeye[i-1] \otimes E_1 \otimes \eeye[k-i]\ +
        \hspace{-.7em} &E_2& \otimes \eeye[i-1] \otimes  E_1
                                \otimes \eeye[k-i]\ + \cdots\ + 
         \hspace{-.7em} &E_d& \otimes \eeye[i-1] \otimes E_1
                                \otimes \eeye[k-i]\  + \\
        &E_1 \otimes \eeye[i-1] \otimes E_2 \otimes \eeye[k-i]\ +
         \hspace{-.7em} &\omega_{d^{i+1}} E_2& \otimes \eeye[i-1] \otimes  E_2
                                \otimes \eeye[k-i]\ +
         \cdots\ +                        
         \hspace{-.7em} & \omega_{d^{i+1}}^{d-1} E_d&\otimes \eeye[i-1] \otimes E_2
                                \otimes \eeye[k-i]\  + \\
         &\cdots\ +\\
         &E_1 \otimes \eeye[i-1] \otimes E_d \otimes \eeye[k-i]\ +
         \hspace{-.7em} &\omega_{d^{i+1}}^{d-1} E_2& \otimes \eeye[i-1] \otimes  E_d
                                \otimes \eeye[k-i]\ +
         \cdots\ +                        
        \hspace{-.7em} &\omega_{d^{i+1}}^{(d-1)^2} E_d& \otimes \eeye[i-1] \otimes E_d
                                \otimes \eeye[k-i],
      \end{aligned} \\[5pt]  
\ = \ & \begin{aligned}[t]
        &(E_1 + \phantom{\omega_{2^{i+1}}} E_2 + \cdots + \phantom{\omega_{d^{i+1}}^{(d-1)^2}} E_d) \otimes
                \eeye[i-1] \otimes E_{1\phantom{-1}} \otimes \eeye[k-i] \ + \\
        &(E_1 + \omega_{d^{i+1}} E_2 + \cdots + \omega_{d^{i+1}}^{d-1\phantom{()^2}} E_d) \otimes
                \eeye[i-1] \otimes E_{2\phantom{-1}} \otimes \eeye[k-i]\ + \\
        & \cdots\ + \\
        &(E_1 + \omega_{d^{i+1}}^{d-1} E_2 + \cdots + \omega_{d^{i+1}}^{(d-1)^2} E_d) \otimes
                \eeye[i-1] \otimes E_{d-1} \otimes \eeye[k-i],
      \end{aligned} \\[5pt]
\ = \ &
\sum_{\ell=1}^{d}
R_{i+1}^{\ell-1} \otimes \eeye[i-1] \otimes E_\ell  \otimes \eeye[k-i],             
\end{align*}
\endgroup
where we subsequently used the distributivity and scalar shift properties
\eqref{eq:kron-prop} of the Kronecker product, and combined the terms with 
 the same matrix, $E_1$, $E_2$ up to $E_d$,
 in the $(i+1)$st position.
\end{proof}

\subsection{Quantum circuit for Radix-d QFT}

The generalization of a qubit for encoding a $d$-dimensional vector
is a \emph{qudit}. When $d=3$, it is called a \emph{qutrit}. 
The implementation of the QFT on a quantum computer working with qudits is studied
in \cite{Cao2011, Stroud2002, Dogra2015}. In the current section, we show that our
derivation of the QFT on a qubit system extends to qudit systems in a straightforward manner,
and leads to a class of equivalent QFT circuits on a qudit system.

Single qudit gates are unitary $d\times d$ matrices and the state space of
multi-qudit systems is $(\C[d])^{\otimes n}$.
A controlled-$U$ operation acting on two qudits is denoted with the same circuit
symbol as for the qudit case, but the action on the target qudit is now dependent
on the $d$ different states of the control qudit:
\[ \Qcircuit @C=1em @R=1em {
   & \ctrl{1} &  \qw\\
   & \gate{U} & \qw}
\raisebox{-1em}
{
$\qquad \leftrightarrow \qquad
(E_1 \otimes \eye[d]) + (E_2 \otimes U) + \cdots + (E_d \otimes U^{d-1}) =: U_c,
$}
\]
with $E_k := e_k e_k\T \inC[d][d]$, for $k=1,\hdots,d$.
This definition is in agreement with  the definition of a controlled-NOT operation
on a $d$-level quantum system, also called a SUM gate \cite{Gottesman}.

\Cref{trdfft} gives the following quantum circuit for the Fourier transform of dimension $d^n \times d^n$
acting on $n$ qudits:
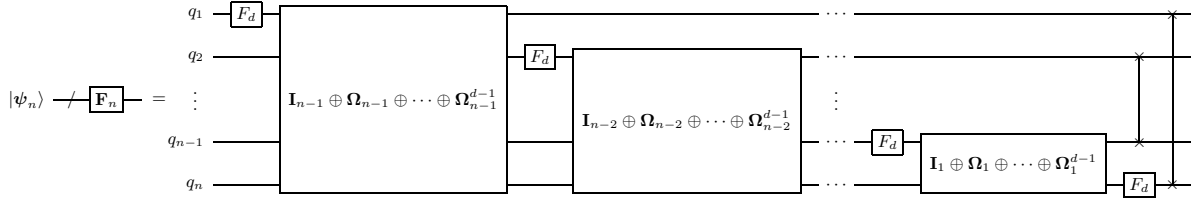
\[\resizebox{.85\hsize}{!}{$
\Qcircuit @C=1em @R=1em {
& & & & & & & \lstick{q_1} & \gate{F_d} & \multigate{4}{\eeye[n-1] \oplus \eOmega_{n-1} \oplus \cdots \oplus \eOmega_{n-1}^{d-1}} &  \qw & \qw & \qw & \cdots & & \qw & \qw & \qw &\qswap & \qw\\
& & & & & & & \lstick{q_2} & \qw & \ghost{\eeye[n-1] \oplus \eOmega_{n-1} \oplus \cdots \oplus \eOmega_{n-1}^{d-1}} & \gate{F_d} & \multigate{3}{\eeye[n-2] \oplus \eOmega_{n-2} \oplus \cdots \oplus \eOmega_{n-2}^{d-1}} & \qw & \cdots & & \qw & \qw & \qswap & \qw & \qw\\
\lstick{\ket{\expsi_n}} &\qw{/} & \gate{\eF_{n}} & \qw  & = & & \raisebox{.5em}{\vdots} & & &  & & \pureghost{\eeye[n-1] \oplus \eOmega_{n-1} \oplus \cdots \oplus \eOmega_{n-1}^{d-1}} & & \raisebox{.5em}{\vdots}\\
& & & & & & & \lstick{q_{n-1}} & \qw & \ghost{\eeye[n-1] \oplus \eOmega_{n-1} \oplus \cdots \oplus \eOmega_{n-1}^{d-1}} & \qw & \ghost{\eeye[n-2] \oplus \eOmega_{n-2} \oplus \cdots \oplus \eOmega_{n-2}^{d-1}} &  \qw & \cdots & & \gate{F_d} & \multigate{1}{\eeye[1] \oplus \eOmega_1 \oplus \cdots \oplus \eOmega_{1}^{d-1}} & \qswap \qwx[-2] & \qw & \qw\\
& & & & & & & \lstick{q_n} & \qw & \ghost{\eeye[n-1] \oplus \eOmega_{n-1} \oplus \cdots \oplus \eOmega_{n-1}^{d-1}} & \qw & \ghost{\eeye[n-2] \oplus \eOmega_{n-2} \oplus \cdots \oplus \eOmega_{n-2}^{d-1}} &  \qw & \cdots &  &\qw & \ghost{\eeye[1] \oplus \eOmega_1 \oplus \cdots \oplus \eOmega_{1}^{d-1}} & \gate{F_d} & \qswap \qwx[-4] & \qw
}
$}\hspace{-10pt}\]
The structure of this circuit is exactly the same as in the qubit case. The difference is that the Hadamard gate $H$ is replaced by the \emph{Fourier}
gate $F_d$ which applies the $d$-dimensional DFT matrix to a single qudit, and that the diagonal operators now consist of a direct sum of $d$ diagonal
matrices.

Just like before, these diagonal operators can be synthesized in controlled-$R_j$ operators acting on two qudits by \Cref{thm:rd-diagdecomp}
and \Cref{lem:dreorder}:
\[
\Qcircuit @C=1em @R=1em {
\lstick{q_1} & \multigate{4}{\eeye[k] \oplus \eOmega_{k} \oplus \cdots \oplus \eOmega_{k}^{d-1}} & \qw  & & & & & & \lstick{q_1} & \ctrl{1} & \qw & & & & & & \lstick{q_1}  & \ctrl{4}  & \gate{R_2} & \qw & \cdots & & \ctrl{3} & \qw\\
\lstick{q_2}  & \ghost{\eeye[k] \oplus \eOmega_{k} \oplus \cdots \oplus \eOmega_{k}^{d-1}} & \qw & & & & & & \lstick{q_2} & \multigate{3}{\eOmega_k} & \qw & & & & & & \lstick{q_2} & \qw & \ctrl{-1} & \qw & \cdots & & \qw & \qw \\
\lstick{\raisebox{.5em}{\vdots}}& \pureghost{\eeye[k] \oplus \eOmega_{k} \oplus \cdots \oplus \eOmega_{k}^{d-1}} & & & = & & & & \lstick{\raisebox{.5em}{\vdots}} & \pureghost{\eOmega_{k}} & && = & & & & \lstick{\raisebox{.5em}{\vdots}} & & & & \raisebox{.5em}{\vdots} & \\
\lstick{q_k} & \ghost{\eeye[k] \oplus \eOmega_{k} \oplus \cdots \oplus \eOmega_{k}^{d-1}} & \qw & & & & &  & \lstick{q_k} & \ghost{\eOmega_{k}} & \qw & & & & & & \lstick{q_k} & \qw & \qw & \qw & \cdots & & \gate{R_k} & \qw\\
\lstick{q_{k+1}}  & \ghost{\eeye[k] \oplus \eOmega_{k} \oplus \cdots \oplus \eOmega_{k}^{d-1}} & \qw & & & & & & \lstick{q_{k+1}} & \ghost{\eOmega_{k}} & \qw & & & & & & \lstick{q_{k+1}} & \gate{R_{k+1}} &\qw  & \qw & \cdots & & \qw & \qw
}
\]
The representation is only one of the many possible implementations of the diagonal matrix under the freedom
allowed by the commutativity and swap property (\Cref{lem:dreorder}) of controlled-$R$ gates.

We conclude that the gate count for a quantum Fourier transform on an $n$ qudit $d$-level system is still $\bigO(n^2)$,
but now for a QFT of a state vector of dimension $d^n$.
This results in a $\log_2 d$ reduction in gates for a qudit system compared to a qubit system 
that has a state space of the same dimension at the cost of more complex controlled operations.

\section{Conclusion}

In this paper, we showed how the quantum Fourier transform algorithm can be 
derived as a decomposition of the discrete Fourier matrix. 
The derivation starts from the radix-2 decomposition of the DFT matrix
that yields the FFT algorithm, and only makes use of Kronecker products
to further decompose a diagonal matrix into a
product of simpler unitary matrices, each of which can be written as the sum 
of two Kronecker products of $2\times 2$ matrices. 
This alternative approach to the derivation of the quantum Fourier transform in \cite{NC2010,Coppersmith1994} requires little knowledge of quantum 
computing.

We showed in section~\ref{sec:rank1} that the Kronecker structure of the QFT decomposition 
does not lead to an immediate reduction in complexity compared to that of
the FFT algorithm on a classical computer,
even when the QFT is applied to a rank-1 tensor. We explained why the 
complexity of the QFT on a quantum computer can be evaluated by counting the
number of $2\times 2$ matrices produced by the QFT matrix 
decomposition, which is $\bigO((\log N)^2)$ for a transformation of a
vector of size $N = 2^n$.
We made the connection between the matrix decomposition of the QFT and
the quantum circuit representation widely used in the quantum computing 
literature. We pointed out that the QFT decomposition of the DFT matrix 
and the corresponding quantum circuit is not unique. The non-uniqueness
of the decomposition can potentially provide some flexibility in 
the quantum circuit topology in terms of the placement of controlled-$R$ gates.
We also generalized the radix-2 decomposition to a radix-$d$ 
decomposition which requires the QFT to be implemented on a quantum
computer equipped with qudits.

Although the QFT can be performed efficiently on a quantum computer, the 
result of the transform is not easily accessible.
This is a key difference between the QFT and FFT, and
between a quantum and classical algorithm in general.
Because of this, one cannot use the QFT in the same way the FFT is used.
For example, it is not straightforward to perform a fast 
convolution of two vectors~\cite{Lomont2003}, which is the most widely used application
of the FFT, on a quantum computer using QFTs.
On the other hand,
the QFT is used as a building block for several quantum algorithms
such as in phase estimation~\cite{Kitaev95, Shor1997}. However, how the QFT is used in these algorithms
is beyond the scope of this paper. 

\section*{Acknowledgments}
This work was supported by the
Laboratory Directed Research and Development Program
of Lawrence Berkeley National Laboratory under
U.S. Department of Energy Contract No. DE-AC02-05CH11231.

\bibliographystyle{plain}
\bibliography{references}

\end{document}